\documentclass[11pt,reqno]{amsart}
\usepackage{amscd,amsthm,amsfonts,amssymb,amsmath,mathtools}
\usepackage[colorlinks=true]{hyperref}
\usepackage{fullpage}
\usepackage[all]{xy}
\usepackage{supertabular}

     \newcommand{\BF}{{\mathbb {F}}}

     \newcommand{\BP}{{\mathbb {P}}}
    \newcommand{\BQ}{{\mathbb {Q}}} \newcommand{\BR}{{\mathbb {R}}}

     \newcommand{\BZ}{{\mathbb {Z}}}

     \newcommand{\CH}{{\mathcal {H}}}

     \newcommand{\fF}{{\mathfrak{F}}}
     \newcommand{\fH}{{\mathfrak{H}}}
     
     \newcommand{\fL}{{\mathfrak{L}}}

    \newcommand{\fQ}{{\mathfrak{Q}}}

    \newcommand{\Gal}{{\mathrm{Gal}}}

    \newcommand{\ord}{{\mathrm{ord}}}

    \renewcommand{\mod}{\ \mathrm{mod}\ }

    \newcommand{\tor}{{\mathrm{tor}}}

    \newcommand*\xbar[1]{%
   \hbox{%
     \vbox{%
       \hrule height 0.5pt % The actual bar
       \kern0.3ex%         % Distance between bar and symbol
       \hbox{%
         \kern -0.05em%      % Shortening on the left side
         \ensuremath{#1}%
         \kern -0.05em%      % Shortening on the right side
       }%
     }%
   }%
}

\DeclareFontFamily{U}{wncy}{}
\DeclareFontShape{U}{wncy}{m}{n}{<->wncyr10}{}
\DeclareSymbolFont{mcy}{U}{wncy}{m}{n}
\DeclareMathSymbol{\Sha}{\mathord}{mcy}{"58}

    \theoremstyle{plain}
    \newtheorem{thm}{Theorem}[section] \newtheorem{cor}[thm]{Corollary}
    \newtheorem{lem}[thm]{Lemma}  \newtheorem{prop}[thm]{Proposition}
    \newtheorem {conj}[thm]{Conjecture}

\theoremstyle{remark} \newtheorem{remark}[thm]{Remark}
\theoremstyle{remark} 
\theoremstyle{remark} 

    \numberwithin{equation}{section}

\begin{document}

\title{The Birch--Swinnerton-Dyer exact formula for quadratic twists of elliptic curves}

\author{Shuai Zhai}

\dedicatory{Dedicated to the memory of John Coates}

\thanks{This work was supported by the Natural Science Foundation of Shandong Province (Grant No. 2022HWYQ-037) and by the Taishan Young Scholar Program (Grant No. tsqn202211043).}

%\thanks{This work was supported by .}

\subjclass[2020]{11F67, 11G05, 11G40.}

\begin{abstract}
In the present paper, we obtain a general lower bound for the $2$-adic valuation of the algebraic part of the central value of the complex $L$-series for the quadratic twists of any elliptic curve over $\BQ$, showing that when the $2$-part of the product of Tamagawa factors grows, the $2$-part of the algebraic central $L$-value grows as well, in accordance with the Birch--Swinnerton-Dyer exact formula. This generalises a result of Coates--Kim--Liang--Zhao to all elliptic curves defined over $\BQ$. We also prove the existence of an explicit infinite family of quadratic twists with analytic rank $0$ for a large family of elliptic curves.
\end{abstract}

\maketitle

%\tableofcontents

\section{Introduction}

Let $E$ be any elliptic curve defined over $\BQ$. We take any global minimal Weierstrass equation for $E$, and write $\Delta_E$ for its discriminant, $\omega_E$ for its N\'{e}ron differential, and $\Omega_E^+$ for the least positive real period of $\omega_E$. Let $L(E, s)$ be the complex $L$-series of $E$. Since $E$ is modular by Wiles' theorem, $L(E, s)$ has an entire analytic continuation, and $L(E, 1)/\Omega_E^+$ is a rational number. Define $c_\infty(E)=\delta_E\Omega_E^+$, where $\delta_E=1$ or $2$ is the number of connected components of $E(\BR)$. For each finite prime $\ell$, let $c_\ell=[E(\BQ_\ell): E_0(\BQ_\ell)]$, where $E_0(\BQ_\ell)$ is the subgroup of $E(\BQ_\ell)$ consisting of points with non-singular reduction modulo $\ell$. Let $\Sha(E)$ denote the Tate--Shafarevich group of $E$. If $L(E,1) \neq 0$, the celebrated results of Gross--Zagier \cite{Gross} and Kolyvagin \cite{Kolyvagin} tell us that both $E(\BQ)$ and $\Sha(E)$ are finite. In addition, the conjecture of Birch and Swinnerton-Dyer predicts that 
\begin{equation}\label{bsd}
\frac{L(E,1)}{c_\infty(E)}=\frac{\prod_\ell c_\ell (E)\cdot |\Sha(E)|}{|E(\BQ)_\mathrm{tor}|^2}.
\end{equation}
Remarkable progress has been made towards the proof of this exact formula by the methods of Iwasawa theory. In particular, the $p$-part of the Birch--Swinnerton-Dyer exact formula for analytic rank 0 and 1 is almost known to hold for all primes $p>2$ due to lots of authors, for example, Rubin \cite{Rubin}, Kato \cite{Kato}, Zhang \cite{Zhang}, Skinner--Urban \cite{Skinner2}, Kobayashi \cite{Kobayashi}, Jetchev--Skinner--Wan \cite{Jetchev}, Wan \cite{Wan}, Castella \cite{Castella1}, et al (see the nice survey article by Coates \cite{Coates2}). Recently, strong $p$-converse theorems of the Birch--Swinnerton-Dyer conjecture have been established by Skinner \cite{Skinner1}, Burungale--Tian \cite{Burungale1} and Castella--Grossi--Lee--Skinner \cite{Castella2}, et al. However, many of the most interesting classical problems involve the small primes $p$ where this formula still has not been established in general, and notably for the prime $p=2$: see, for example, the remarkable work of Tian \cite{Tian1}\cite{Tian2} and Tian--Yuan--Zhang \cite{Tian3} on the congruent number problem, and Smith \cite{Smith} and Kriz--Li \cite{Kriz} on Goldfeld's conjecture \cite{Goldfeld}. Very recently, Burungale--Flach \cite{Burungale2} proved the validity of \eqref{bsd} for all elliptic curves with complex multiplication, and furthermore, gave a complete proof of its equivariant refinement formulated by Gross. In the present paper, we use elementary methods, which involve no Iwasawa theory, to prove both some lower bound results of the 2-part of \eqref{bsd}, and also some special cases of the 2-part of \eqref{bsd}, when $E$ runs over a large family of quadratic twists of some fixed elliptic curve defined over $\BQ$.

We write $ord_2$ for the $2$-adic valuation on $\BQ$, normalised so that $ord_2(2) = 1$ and $ord_2(0) = \infty$. In all that follows, $M$ will always denote a square free positive or negative integer such that $M\equiv 1 \mod 4$ with $M \neq 1$, and we shall write $E^{(M)}$ for the twist of $E$ by the extension $\BQ(\sqrt{M})/\BQ$. For any odd prime $q \mid M$, we define 
$$
t_E(M):= \sum_{q \mid M} t(q), \text{ where } t(q) = 
\left\{
\begin{array}{ll}
1            & \hbox{if $2 \mid c_q(E^{(M)})$;} \\
0            & \hbox{if $2 \nmid c_q(E^{(M)})$.}
\end{array}
\right.
$$
We recall that $E$ is said to be optimal if the map from the modular curve $X_0(C)$ to $E$ does not factor through any other elliptic curve defined over $\BQ$. We now state the main result of our paper.

\begin{thm}\label{ThmLowerBound}
Let $E$ be any optimal elliptic curve over $\BQ$ having conductor $C$. Then, for all square free integers $M \equiv1 \mod 4$ with $(M, C) = 1$ such that $L(E^{(M)}, 1) \neq 0$, we have
\begin{equation}\label{lower}
ord_2(L(E^{(M)},1)/c_\infty(E^{(M)})) \geq t_E(M)-1-ord_2(\nu_E), 
\end{equation}
where $\nu_E$ is the Manin constant of $E$.
\end{thm}

Note that, if we assume $\nu_E$ is odd, the theorem gives the lower bound $t_E(M)-1$. In fact, it has been shown that $\nu_E$ is odd for $4 \nmid C$ by the work of Mazur \cite{Mazur}, Abbes--Ullmo \cite{Abbes}, Agashe--Ribet--Stein \cite{Agashe} and \v{C}esnavi\v{c}ius \cite{Cesnavicius}. In addition, Cremona \cite{Cremona2} has shown numerically that $\nu_E$ is $1$ for $C \leq 60000$. Previously, Zhao \cite{Zhao1}\cite{Zhao2}\cite{Zhao3} proved several results like \eqref{lower} for the congruent number family of elliptic curves, which were subsequently used by Tian in his induction arguments on the congruent number problem. Coates--Kim--Liang--Zhao \cite{Coates3} proved an analogue of \eqref{lower} for a wide class of elliptic curves with complex multiplication, which was then applied in \cite{Coates4} to prove a generalisation of Birch's lemma, and the $2$-part of the Birch--Swinnerton-Dyer conjecture for quadratic twists of $X_0(49)$. Kezuka \cite{Kezuka1}\cite{Kezuka2}\cite{Kezuka3} and Kezuka--Li \cite{Kezuka4} generalised Zhao's method from $2$-adic to $3$-adic and proved analogous results for the family of cubic twists of the modular curve $X_0(27)$. These earlier methods work only for certain families of elliptic curves with complex multiplication, since they make use of the fact that the value at $s=1$ of the complex $L$-series of such a curve is a sum of Eisenstein series. 

Previously, we have successfully applied the modular symbols on the 2-part of the central $L$-values, but the early method was only valid for certain quadratic twists of certain elliptic curves \cite{Cai}\cite{Zhai2}. One main reason is that, for all elliptic curves, both the real and imaginary periods are getting involved, and the earlier integrality argument therefore failed. In order to apply it on all elliptic curves, we construct two integrality arguments separately, and successfully apply them at the same time in a complete induction argument. Another reason is that, as the given elliptic curve varies, the $2$-part of the $L$-value also varies, making it difficult to control the initial $L$-value in the induction argument. However, we could obtain a uniform lower bound, which turns out to be enough to work on all elliptic curves when combined with our new integrality arguments. Thus, the new techniques used in this paper are valid for all quadratic twists of any elliptic curve over $\BQ$. 

Note that Theorem \ref{MainThm-1} tells us that the lower bound appeared in Theorem \ref{ThmLowerBound} is generally the best, since the equality holds for the family of elliptic curves in Theorem \ref{MainThm-1}. Since there are only finitely number of bad primes of $E$, we immediately have the following result.
\begin{cor}\label{CorLowerBound}
Let $E$ be any elliptic curve over $\BQ$ having conductor $C$. Then, for all square free integers $M'$ such that $L(E^{(M')}, 1) \neq 0$, we have
\begin{equation}\label{lower2}
ord_2(L(E^{(M')},1)/c_\infty(E^{(M')})) \geq t_E(\epsilon M')-T_{\epsilon C}-ord_2(\nu_E), 
\end{equation}
where $\epsilon=\pm 1, 1/2$ such that $\epsilon M' \equiv 1 \mod 4$, and $T_{\epsilon C}$ is an integer only related to $\epsilon C$.
\end{cor}

We now give some non-vanishing results for quadratic twists of certain elliptic curves defined over $\BQ$. The following conjecture is folklore.
\begin{conj}\label{Conj}
For any elliptic curve $E$ over $\BQ$, and any positive integer $r$, there are infinitely many square-free integers $M$, having exactly $r$ prime factors, such that $L(E^{(M)}, 1) \neq 0.$ 
\end{conj}
We shall verify some special cases of this conjecture for certain elliptic curves $E/\BQ$ with $L(E, 1) \neq 0$. If $q$ is any prime of good reduction for $E$, we write $a_q$ for the trace of Frobenius at $q$ on $E$, so that $|E(\BF_q)|=1+q-a_q$ is the number of $\BF_q$-points on the reduction of $E$ modulo $q$. For each integer $n > 1$, we let $E[n]$ be the Galois module of $n$-division points on $E$. 
Assume $E(\BQ)[2] \leq \BZ/2\BZ$, we define a set of primes given by 
\begin{equation}\label{eqS}
\mathcal{S}=
\left\{
\begin{array}{ll}
\{q \equiv 1,3 \mod 4 \ | \ ord_2(|E(\BF_q)|)=ord_2(|E(\BQ)[2]|)\}         & \hbox{if $\Delta_E<0$;} \\
\{q \equiv 1\ \  \ \mod 4 \  | \ ord_2(|E(\BF_q)|)=ord_2(|E(\BQ)[2]|)\}     & \hbox{if $\Delta_E>0$.}
\end{array}
\right.
\end{equation}
We have the following result.
\begin{thm}\label{ThmNonvanishing}
Let $E$ be an optimal elliptic curve over $\BQ$ such that 
\begin{enumerate}
  \item $E(\BQ)[2] \leq \BZ/2\BZ$;
  \item $ord_2 (L(E,1)/c_\infty(E))=-ord_2(|E(\BQ)[2]| \cdot \nu_E)$.
\end{enumerate}
Let $M$ be an integer congruent to $1$ modulo $4$ having all prime factors lying in $\mathcal{S}$. Then we have 
$$
\ord_{s=1}L(E^{(M)}, s)=\mathrm{rank}\  E^{(M)}(\BQ)=0.
$$
Moreover, the Tate--Shafarevich group $\Sha(E^{(M)})$ is finite.
\end{thm}
\begin{proof}
When $E(\BQ)[2] \cong \BZ/2\BZ$, the theorem follows from Theorem \ref{MainThm-1}, and \cite[Theorem 1.1]{Cai} (in particular for $\Delta_E>0$). When $E(\BQ)[2]$ is trivial, the theorem follows from \cite[Theorem 1.1 \& Theorem 1.3]{Zhai1}.
\end{proof}

In Section 2, we shall investigate under what conditions there will be a positive density of primes in the set $\mathcal{S}$. It turns out that, in view of Corollary \ref{cor3}, $\mathcal{S}$ is always a set of positive density of primes when $E(\BQ)[2]=0$. When $E(\BQ)[2] \cong \BZ/2\BZ$, we define $E':=E/E(\BQ)[2]$ to be the $2$-isogenous curve of $E$ under the natural $2$-isogeny, and write $\phi: E \to E'$ for the corresponding 2-isogeny defined over $\BQ$. Define $F = \BQ(E[2])$ and $F' = \BQ(E'[2])$. Thus $[F:\BQ] = 2$, but we have either $E'(\BQ)[2]=\BZ/2\BZ$ or $E'(\BQ)[2]=\BZ/2\BZ \times \BZ/2\BZ$, so that either $[F':\BQ] = 2$ or $F' = \BQ$. In view of Corollary \ref{cor1} and Corollary \ref{cor2}, we see that $\mathcal{S}$ is a set of positive density of primes if and only if $F' \neq \BQ$ when $\Delta_E<0$, and $F' \neq \BQ, \BQ(i)$ when $\Delta_E>0$. If $\mathcal{S}$ is infinite, we can verify Conjecture \ref{Conj} for certain elliptic curves.

\begin{cor}\label{ThmConj}
Let $E$ be an elliptic curve as in Theorem \ref{ThmNonvanishing}. Conjecture \ref{Conj} is true, provided $\mathcal{S}$ is infinite.
\end{cor}

\begin{remark}
The conclusions of Theorem \ref{ThmNonvanishing} and Corollary \ref{ThmConj} also hold for the quadratic twists of elliptic curves lying in the same isogeny class as $E$, since they have the same complex $L$-function as the corresponding quadratic twists of $E$. We explain in the next section, using the Chebotarev density theorem, why $\mathcal{S}$ is usually non-empty, and how to construct an explicit infinite set primes with positive density lying in it in most cases. Also, the prime factors of $M$ can be any primes in $\mathcal{S}$. Note that when $\Delta_E>0$, we assume that  $q \equiv 1 \mod 4$, since our argument in the proof of Theorem \ref{ThmNonvanishing} does not work for primes $q$ with $q \equiv 3 \mod 4$ if $\Delta_E>0$. Indeed, when $q \equiv 3 \mod 4$, one has $L(E^{(-q)},1)=0$ for some $E$ and some primes $q$ satisfying all the other conditions in the definition of the set $\mathcal{S}$, for example, using the labelling of curves introduced by Cremona \cite{Cremona1}, $34a1$ with $q = 3$, $99c1$ with $q = 7$.
\end{remark}

\begin{remark}
Theorem \ref{ThmNonvanishing} can be applied to the family of quadratic twists of many elliptic curves $E/\BQ$. In particular, it can be applied to the following $E$ from Cremona's Tables \cite{Cremona1} with conductor $C < 100$: $11a1$, $14a1$, $19a1$, $20a1$, $26a1$, $27a1$, $34a1$, $35a1$, $36a1$, $37b1$, $38a1$, $38b1$, $44a1$, $46a1$, $49a1$, $50a1$, $50b1$, $51a1$, $52a1$, $54a1$, $54b1$, $56b1$, $66a1$, $66c1$, $67a1$, $69a1$, $73a1$, $75a1$, $75c1$, $76a1$, $77c1$, $80b1$, $84a1$, $84b1$, $89b1$, $92a1$, $94a1$, $99c1$, $99d1$.
\end{remark}

The result for $E(\BQ)[2] \cong \BZ/2\BZ$ in Theorem \ref{ThmNonvanishing} has been applied in a recent paper \cite{Shu} with Shu, which shows parallel results such that the family of quadratic twists of elliptic curves has a rational point of infinite order and then verify the $2$-part of the Birch and Swinnerton-Dyer conjecture for those curves. For some recent results towards this direction, one can see a nice survey article by Li \cite{Li}.
\medskip

{\bf Acknowledgments.} This paper is dedicated to the cherished memory of my supervisor, John Coates, who passed away in 2022. I am very grateful to him for introducing this research topic to me and for inspiring my deep passion for the Birch--Swinnerton-Dyer conjecture. I sincerely thank him for his constant encouragement, insightful advice, and invaluable guidance, which greatly shaped my academic journey. Though his absence is deeply felt, his legacy continues to inspire my work and countless others who had the privilege of learning from him.

\medskip

\section{The infinitude of $\mathcal{S}$}

We show in this section that, in most cases, the set of primes $\mathcal{S}$ has positive density in the set of all rational primes, by using the Chebotarev density theorem. If $q$ is an odd prime of good reduction of $E$, then multiplication by $2$ is an automorphism of the formal group of $E$ at $q$, and so reduction modulo $q$ gives an isomorphism from the $2$-primary subgroup of $E(\BQ_q)$ onto the $2$-primary subgroup of $E(\BF_q)$. Recall that when $E(\BQ)[2] \cong \BZ/2\BZ$, write $\phi: E \to E'$ for the $2$-isogeny defined over $\BQ$. 

\begin{lem}\label{lem0} 
Let $q$ be an odd prime of good reduction for $E$. (i) If $E(\BQ)[2] = 0$, then $E(\BQ_q)[2] = 0$ if and only if $[\BQ_q(E[2]):\BQ_q] = 3$. (ii) If $E(\BQ)[2] = \BZ/2\BZ$, then $E(\BQ_q)[4] = \BZ/2\BZ$ if and only if both $[\BQ_q(E[2]):\BQ_q] = 2$   and also $[\BQ_q(E'[2]):\BQ_q] = 2$.
\end{lem}
\begin{proof} Assuming $q$ is an odd prime of good reduction, then $\BQ_q(E[2])$ is an unramified, and therefore cyclic extension of $\BQ_q$, whose Galois group is a quotient group of $\Gal(\BQ(E[2])/\BQ)$. But $\Gal(\BQ(E[2])/\BQ)$ is naturally isomorphic to a subgroup of $GL(2, \BF_2)$ and $\Gal(\BQ_q(E[2])/\BQ_q)$ is therefore cyclic of order $1$, $2$, or $3$. Now we can choose an equation for $E$ of the form $y^2 = f(x)$, where $f(x)$ is a polynomial in $\BZ[x]$ of degree $3$, which has good reduction at $q$. Then the non-zero points of order $2$ on $E$ are given by the points $(0,\alpha)$, where $\alpha$ runs over the three roots of $f(x)$. Thus it is clear that $E(\BQ_q)[2] = 0$ if and only if $f(x)$ is a monic irreducible modulo $q$, or equivalently $[\BQ_q(E[2]):\BQ_q] = 3$, proving (i).  Assume now that $E(\BQ)[2] = \BZ/2\BZ$. Suppose first that $E(\BQ_q)[4] = \BZ/2\BZ$. Then  $E(\BQ_q)[2] = \BZ/2\BZ$, and so $[\BQ_q(E[2]):\BQ_q] = 2$. Moreover, since $|E'(\BF_q)|=|E(\BF_q)|$, we must have $E'(\BQ_q)[4] = E'(\BQ_q)[2] = \BZ/2\BZ$, and so $[\BQ_q(E'[2]):\BQ_q] = 2$. Conversely, assume that $[\BQ_q(E[2]):\BQ_q] = 2$ and $[\BQ_q(E'[2]):\BQ_q] = 2$. If, on the contrary, there exists a point $P$ in $E(\BQ_q)$ which is of exact order $4$, then $P' = \phi(P)$ would be a point in $E'(\BQ_q)$ of order dividing $4$. Moreover, if we write $\phi': E' \to E$ for the isogeny dual to $\phi$, then $\phi'(P') = 2P \neq 0$. It follows that $E'(\BQ_q)[4]$ contains an element which is not in the kernel of the isogeny $\phi'$, and so we must have $E'(\BQ_q)[2]$ necessarily has order $4$, contradicting our assumption. This completes the proof. 
\end{proof}

\begin{cor}\label{cor3} 
Assume $E(\BQ)[2]=0$. (i) There is a positive density of odd primes of good reduction $q$ such that $E(\BQ_q)[2] = 0$.
(ii) There is a positive density of odd primes of good reduction such that both $q \equiv 1 \mod 4$ and $E(\BQ_q)[2] = 0$.
\end{cor}
\begin{proof} 
Let $H=\BQ(E[2])$. Assertion (i) follows easily on applying the Chebotarev density theorem to $\Gal(H/\BQ)$, recalling that this Galois group is either cyclic of order $3$, or the symmetric group on $3$ elements. Turning to assertion (ii), and in what follows we write $I = \BQ(i)$. Denote $\fH = HI$, $G = \Gal(\fH/\BQ)$. Suppose first that $\Gal(H/\BQ)$ is cyclic of order $3$, so that $G$ is a cyclic group of order $6$. Let $\fQ$ be the set of all odd primes of good reduction $q$ whose Frobenius element in $G$ is of exact order $3$. Then $q \in \fQ$ must split in the quadratic extension $I/\BQ$, and the assertion follows immediately by the Chebotarev density theorem. When $\Gal(H/\BQ)$ is the symmetric group on 3 elements, $G$ will have order $6$ or $12$. Again we simply take an element $g \in G$ of exact order $3$, and let $\fQ$ be the set of primes of odd good reduction whose Frobenius elements  in $G$ lie in the conjugacy class of $g$. Since $g$ has order 3, the primes in $\fQ$ must split in the quadratic extension $I/\BQ$, and the assertion again follows from the Chebotarev density theorem. This completes the proof.
\end{proof}

When $E(\BQ)[2] \cong \BZ/2\BZ$, recall that $\Delta_E$ is the minimal discriminant of $E$, and we also we write $\Delta_{E'}$ for the minimal discriminant of $E'$. Recall that $F = \BQ(E(\BQ)[2])=\BQ(\sqrt{\Delta_E})$ and $F' = \BQ(E'(\BQ)[2])=\BQ(\sqrt{\Delta_{E'}})$. Of course, the prime factors of both $\Delta_E$ and $\Delta_{E'}$ are precisely the set of primes of bad reduction, but the exponents to which these primes occur are usually different for $E$ and $E'$, and the signs of $\Delta_E$ and $\Delta_{E'}$ may also differ.  

 \begin{cor}\label{cor1}
Assume $E(\BQ)[2]=\BZ/2\BZ$. Then there is a positive density of odd primes of good reduction $q$ such that $E(\BQ_q)[4]= \BZ/2\BZ$ if and only if $F' \neq \BQ$. 
\end{cor}
\begin{proof} Note that $[F:\BQ] = 2$. If $F' = \BQ$, Lemma \ref{lem0} shows that there is no odd good prime $q$ such that 
 $E(\BQ_q)[4] = \BZ/2\BZ$.
Suppose next that $[F': \BQ] = 2$. Hence we are seeking odd primes $q$ of good reduction which are inert in both  $F$ and $F'$. If $F = F'$, there is clearly a set of positive density of such primes $q$, and so we can assume that $F \neq F'$. Then the compositum $J = FF'$ will be a biquadratic extension of $\BQ$, and we let $K$ be the third quadratic extension of $\BQ$ contained in $J$. Let $\fQ$ be the set of all odd primes of good reduction $q$ which have a prime of $K$ lying above them which is inert in $J$. Since $J = FK = F'K$ is an extension of $\BQ$ whose Galois group is isomorphic to the non-cyclic group $\BZ/2\BZ \times \BZ/2\BZ$, we see that every prime $q \in \fQ$ must split completely in $K$. In turn, this implies that every prime $q \in \fQ$ must be inert in both $F$ and $F'$. Now the set $\fQ$ has positive density in the set of all primes by the Chebotarev theorem, and the proof of (ii) is now follows from (ii) of Lemma \ref{lem0}.
\end{proof}

\begin{cor}\label{cor2} 
Assume that  $E(\BQ)[2]=\BZ/2\BZ$ and we have both $F \neq \BQ(i)$ and $F' \neq \BQ, \BQ(i)$. Then there is a set of positive density of odd primes $q$ of good reduction such that both $q \equiv 1 \mod 4$ and
$E(\BQ_q)[4]= \BZ/2\BZ$.
\end{cor}
\begin{proof} We use a similar argument to that used in the proof of Corollary \ref{cor1}. First suppose that $F = F'$. Then the compositum $\fF = FI$ is a biquadratic extension of $\BQ$, and we take $\fQ$ to be the set of odd primes $q$ of good reduction which have a prime of $I$ lying above them which is inert in $\fF$.
Then all primes in $\fQ$ must split completely in $I$, and must be inert in the field $F$, proving the corollary in this case. Hence we can assume that $F \neq F'$, so that $J = FF'$ is a biquadratic extension of $\BQ$. Let $K$ be the third quadratic extension of $\BQ$ contained in $J$. If $K = I$, exactly the same argument as in the proof of Corollary \ref{cor1} proves the assertion here. Hence we can assume that $K \neq I$, whence the field $D = JI$ has Galois group over $\BQ$ isomorphic to $(\BZ/2\BZ)^3$. Then $D$ is a quadratic extension of the field $KI$. We now take $\fQ$ to be the set of odd primes $q$ of good reduction which have a prime of $KI$ lying above them which is inert in $D$. By the Chebotarev theorem, $\fQ$ has positive density in the set of all rational primes. Since unramified extensions of $\BQ_q$ must have cyclic Galois groups, we conclude that all primes in $\fQ$ must split completely in the field $KI$. Moreover, as $D = KIF = KIF'$, we see that all primes in $\fQ$ must then be inert in $F$ and $F'$, completing the proof of the corollary.
\end{proof}

\medskip

\section{Results from modular symbols}

Throughout this section, we shall review some results from modular symbols, and establish a new integrality argument at $2$. We always assume that $E$ is indeed optimal.

\subsection{Modular symbols}
Recall that $C$ denotes the conductor of $E$, and we always write $m$ for an odd positive square-free integer with $(m,C)=1$. Let $r(m)$ be the number of distinct prime factors of $m$. Write $f$ for the modular form attached to $E$, and $a_m$ for the Fourier coefficient of the modular form $f$ attached to $E$. Let $\CH$ denote the upper half plane, and put $\CH^*:=\CH\cup\BP^1(\BQ)$ and $X_0(C):=\Gamma_0(C) \backslash \CH^*$. Let $\alpha, \beta$ be two points in $\CH^*$ with $\beta=g\alpha$, $g\in \Gamma_0(C)$, and denote $\{\alpha,\beta\}$ to be an integral homology class in $H_1(X_0(C),\BZ)$ from $\alpha$ to $\beta$. Denote $\langle \{\alpha,\beta\}, f \rangle := \int_{\alpha}^{\beta} 2\pi i f(z) dz$.

For any positive integer $l \mid m$, we denote $\chi_l$ to be the primitive quadratic character modulo $l$, and we define
$$
\langle m \rangle_{\chi_l} := \sum_{k \in (\BZ/m\BZ)^{\times}} \chi_l(k) \langle \{0,\frac{k}{m}\}, f \rangle.
$$
%In particular, we have $\langle m \rangle_{\chi_1}=\langle m \rangle_{\chi_l^0}$ for any $l \mid m$. For simplicity, we use $\langle m \rangle_{\chi_1}$ instead of both of them if there is no danger of confusion.

\begin{prop} 
We have
\begin{equation}\label{ms1}
(\sum_{l \mid m}l-a_m) L(E,1)=-\sum_{\substack{l \mid m \\ k \mod l}} \langle \{0,\frac{k}{l}\}, f \rangle = -\sum_{j=1}^{r(m)} 2^{r(m)-j} \sum_{\substack{d \mid m \\ r(d)=j}} \langle d \rangle_{\chi_1},
\end{equation}
where $l$ runs over all positive divisors of $m$; and
\begin{equation}\label{ms2}
L(E,\chi_m,1)=\frac{g(\xbar{\chi_m})}{m} \sum_{k \mod m}{\chi_m}(k) \langle \{0,\frac{k}{m}\}, f \rangle = \frac{g(\xbar{\chi_m})}{m} \cdot \langle m \rangle_{\chi_m},
\end{equation}
where
$$
g(\xbar{\chi_m}):= \sum_{k \mod m} \xbar{\chi_m}(k) e^{2\pi i \frac{k}{m}}.
$$ 

\end{prop}
\begin{proof}
See Manin \cite[Theorem 4.2]{Manin}, or Cremona \cite[Chapter 3]{Cremona1}.
\end{proof}

In what follows, we always denote $N_q:=|E(\BF_q)|=1+q-a_q$. In the following lemma, we represent $N_{q_1} N_{q_2} \cdots N_{q_{r(m)}} L(E,1) $ in two different ways, which will be applied to obtain our main results according as the prime factors of $m$ are congruent to $3$ or $1$ modulo $4$, respectively.
\begin{prop} 
We have
\begin{equation}\label{NqL(E,1)'}
N_{q_1} N_{q_2} \cdots N_{q_{r(m)}} L(E,1) = - \sum_{j=1}^{r(m)} \sum_{\substack{d \mid m \\ r(d)=j}}  \prod_{q \mid \frac{m}{d}} (1 - q + N_q) \cdot \langle d \rangle_{\chi_1};
\end{equation}
\begin{equation}\label{NqL(E,1)}
N_{q_1} N_{q_2} \cdots N_{q_{r(m)}} L(E,1) = (-1)^{r(m)} \sum_{j=1}^{r(m)} \sum_{\substack{d \mid m \\ r(d)=j}} \prod_{q \mid \frac{m}{d}} (1 - q) \cdot \langle d \rangle_{\chi_1};
\end{equation}
\begin{equation}\label{T'_{d,m}}
\langle m \rangle_{\chi_d}  = (a_q - 2 \chi_d(q)) \langle m/q \rangle_{\chi_d} = \prod_{q \mid \frac{m}{d}}(a_q - 2 \chi_d(q)) \cdot \langle d \rangle_{\chi_d},
\end{equation}
where $r(m) \geq 2$, $d>1$ is a positive integer dividing $m$, and $q$ is a prime dividing $\frac{m}{d}$.
\end{prop}
\begin{proof}
We now prove \eqref{NqL(E,1)'} by induction on $r(m)$, the number of prime factors of $m$. The assertion is true for $r(m)=1$ by \eqref{ms1}. Assume next that $r(m)=2$. Note that
$$
N_{q_1} N_{q_2} = (a_{q_1}+N_{q_1})(a_{q_2}+N_{q_2}) - a_{q_1}a_{q_2} - (a_{q_2}N_{q_1}+a_{q_1}N_{q_2}),
$$
and also note $a_q+N_q=1+q$ by the definition of $a_q:=1+q-N_q$. Thus
$$
N_{q_1} N_{q_2} L(E,1) = - \sum_{l \mid q_1 q_2} \langle l \rangle_{\chi_1} + a_{q_2} \langle q_1 \rangle_{\chi_1} +a_{q_1} \langle q_2 \rangle_{\chi_1} .
$$
Then by \eqref{ms1}, we have
\begin{align*}
N_{q_1} N_{q_2} L(E,1) &= (a_{q_2}-2)\langle q_1 \rangle_{\chi_1}  + (a_{q_1}-2)\langle q_2 \rangle_{\chi_1}  - \langle q_1 q_2 \rangle_{\chi_1}  \\
                       &= -\left((1-q_2+N_{q_2}) \langle q_1 \rangle_{\chi_1}  + (1-q_1+N_{q_1}) \langle q_2 \rangle_{\chi_1} + \langle q_1 q_2 \rangle_{\chi_1}\right),
\end{align*}
as required. Now assume that $r(m) > 2$, and that the lemma is true for all divisors $n>1$ of $m$ with $n \neq m$. We then consider the case $m=q_1 q_2 \cdots q_{r(m)}$. First note that
\begin{equation}\label{eqN}
\begin{split}
N_{q_1} N_{q_2} \cdots N_{q_{r(m)}} =& (a_{q_1}+N_{q_1})(a_{q_2}+N_{q_2}) \cdots (a_{q_{r(m)}}+N_{q_{r(m)}}) - a_{q_1}a_{q_2} \cdots a_{q_{r(m)}} \\
                                     & - (\sum_{i=1}^{r(m)}N_{q_i} \prod_{\substack{k=1 \\ k \neq i}}^{r(m)}a_{q_k} + \sum_{i,j=1}^{r(m)}N_{q_i}N_{q_j} \prod_{\substack{k=1 \\ k \neq i,j}}^{r(m)}a_{q_k} + \cdots + \sum_{i=1}^{r(m)}a_{q_i} \prod_{\substack{k=1 \\ k \neq i}}^{r(m)}N_{q_k}).
\end{split}
\end{equation}
Without loss of generality, here we can just consider the coefficients of $\langle q_1 \rangle_{\chi_1}$, $\langle q_1 q_2 \rangle_{\chi_1}$, $\ldots$, $S'_{q_1 q_2 \cdots q_{r(m)}}$ in the identity of the lemma, say, $b'_{q_1}$, $b'_{q_1 q_2}$, $\ldots$, $b'_{q_1 q_2 \cdots q_{r(m)}}$, respectively. By our assumption, we conclude that
$$
b'_{q_1} = -2^{r(m)-1} + \sum_{i=2}^{r(m)}a_{q_i} \prod_{\substack{k=2 \\ k \neq i}}^{r(m)}(2-a_{q_k}) + \sum_{\substack{i,j=2 \\ i \neq j}}^{r(m)} a_{q_i}a_{q_j} \prod_{\substack{k=2 \\ k \neq i,j}}^{r(m)}(2-a_{q_k}) + \cdots + a_{q_2} a_{q_3} \cdots a_{q_{r(m)}}.
$$
Note that
$$
-2^{r(m)-1} = - \prod_{i=2}^{r(m)}(a_{q_i} + (2-a_{q_i})),
$$
hence we have
$$
b'_{q_1} = - \prod_{i=2}^{r(m)}(2-a_{q_i}) = - \prod_{i=2}^{r(m)}(1-q_i+N_{q_i}).
$$
In view of \eqref{eqN}, similar arguments hold for $b'_{q_1 q_2}$, $\ldots$, $b'_{q_1 q_2 \cdots q_{r(m)-1}}$, and it is easy to see that
$$
b'_{q_1 q_2 \cdots q_{r(m)}} = -1.
$$
The proof of \eqref{NqL(E,1)'} is complete. For the proof of \eqref{NqL(E,1)} and \eqref{T'_{d,m}}, see \cite[Lemma 2.1, Lemma 4.1]{Cai}.
\end{proof}

\smallskip

\subsection{Period lattice}
Let $\omega_E$ be a N\'{e}ron differential on a global minimal Weierstrass equation for $E$, and $\fL_E$ be the period lattice of $\omega_E$. Recall that $\nu_E$ denotes the Manin constant of $E$. Denote $\Omega_E^+$ ($i \Omega_E^-$) to be the least positive real (purely imaginary) period of $\omega_E$, and similarly $\Omega_f^+$ ($i \Omega_f^-$) to be the least positive real (purely imaginary) period of $f$, which are chosen such that $\nu_E=\Omega_E^+/\Omega_f^+=\Omega_E^-/\Omega_f^-$. Recall that when $\Delta_E<0$ (resp. $\Delta_E>0$), the period lattice $\fL_E$ has a $\BZ$-basis of the form $\left[\Omega_E^+, (\Omega_E^+ + i \Omega_E^-)/2\right] \ \left(\text{resp.} \ [\Omega_E^+, i \Omega_E^-]\right)$, where $\Omega_E^+$ and $\Omega_E^-$ are both real, and the period lattice $\Lambda_f$ of $f$ has a $\BZ$-basis of the form $\left[\Omega_f^+, (\Omega_f^+ + i \Omega_f^-)/2\right] \ \left(\text{resp.} \ [\Omega_f^+, i \Omega_f^-]\right)$, where $\Omega_f^+$ and $\Omega_f^-$ are also both real. Thus, we have $c_\infty(E):=\delta_E\Omega_E^+$, $c^-_\infty(E) := \delta_E \Omega_E^-$, and  $c_f := \delta_E \Omega_f^+ = c_\infty(E)/\nu_E$, ${c^-_f} := \delta_E \Omega_f^- = c^-_\infty(E)/\nu_E$, where $\delta_E$ is the number of connected components of $E(\BR)$. Thus, we can write 
\begin{equation}\label{k/mf}
\langle \{0,\frac{k}{m}\}, f \rangle = (s_{k,m} c_f + i t_{k,m} {c^-_f})/2,
\end{equation}
where $s_{k,m}, t_{k,m}$ are integers depending on $f$. In particular, $s_{k,m}, t_{k,m}$ are of the same parity when $\Delta_E<0$. 
\begin{prop}\label{m_chi}
We have 
\begin{equation}\label{S'm}
\langle m \rangle_{\chi_1}/c_f = \sum_{\substack{k=1 \\ (k,m)=1}}^{(m-1)/2} s_{k,m};
\end{equation}
\begin{equation}\label{Tm}
\langle m \rangle_{\chi_m}/c_f=\sum_{\substack{k=1 \\ (k,m)=1}}^{(m-1)/2} \chi_{m}(k) s_{k,m} \ (\text{if } m \equiv 1 \mod 4);
\end{equation}
\begin{equation}\label{Tm-2}
\langle m \rangle_{\chi_m}/i{c^-_f}=\sum_{\substack{k=1 \\ (k,m)=1}}^{(m-1)/2} \chi_{m}(k) t_{k,m} \ (\text{if } m \equiv 3 \mod 4).
\end{equation}
\end{prop}
\begin{proof}
When $m \equiv 1 \mod 4$, we have $\chi_{m}(k)=\chi_{m}(m-k)$, and when $m \equiv 3 \mod 4$, we have $\chi_{m}(k)=-\chi_{m}(m-k)$. Note that $\langle \{0,\frac{k}{m}\}, f \rangle$ and $\langle \{0,\frac{m-k}{m}\}, f \rangle$ are complex conjugate periods of $f$, the assertions follow immediately in view of \eqref{k/mf}.
\end{proof}

Moreover, combining \eqref{ms1} with \eqref{S'm}, we have the following result.
\begin{cor}
We have 
\begin{equation}\label{N_qL(E,1)}
ord_2(N_q L(E,1) / {c_f}) \geq 0,
\end{equation}
for any odd prime $q$ with $(q,C)=1$. 
\end{cor}

\begin{cor}\label{N_qL(E,1)m}
Let $m$ be any integer of the form $m=q_1 q_2 \cdots q_{r(m)}$, with $(m,C)=1$, $r(m) \geq 1$, and $q_1, \ldots, q_{r(m)}$ arbitrary distinct odd primes congruent to $3$ modulo $4$. If $ord_2 (N_{q_i}) = 1$ for any $1 \leq i \leq r(m)$, then we have
$$
ord_2 (N_{q_1} N_{q_2} \cdots N_{q_{r(m)}} L(E,1) / c_f) = ord_2 (\langle m \rangle_{\chi_1} / c_f).
$$
\end{cor}
\begin{proof}
We prove this lemma by induction on $r(m)$, the number of prime factors of $m$. The assertion is obviously true for $r(m)=1$ according to \eqref{ms1}. When $r(m)=2$, say $m=q_1 q_2$, in view of \eqref{NqL(E,1)'}, we have that
\begin{equation}\label{N_q_1_2}
N_{q_1} N_{q_2} L(E,1) =  -\left((1-q_2+N_{q_2})\langle q_1 \rangle_{\chi_1} + (1-q_1+N_{q_1})\langle q_2 \rangle_{\chi_1} + \langle q_1 q_2 \rangle_{\chi_1}\right).
\end{equation}
The assertion for $r(m)=2$ then follows by noting that $1-q_i+N_{q_i} \equiv 0 \mod 4$ and 
$$
ord_2((1-q_2+N_{q_2})\langle q_1 \rangle_{\chi_1} / c_f) = ord_2((1-q_1+N_{q_1})\langle q_2 \rangle_{\chi_1} / c_f) > ord_2(N_{q_1} N_{q_2} L(E,1) / c_f).
$$
Now assume that $r(m) > 2$, and that the lemma is true for all divisors $n>1$ of $m$ with $n \neq m$. We then consider the case $m=q_1 q_2 \cdots q_{r(m)}$. Again by \eqref{NqL(E,1)'}, we have that
\begin{equation}\label{N_q_1_m}
N_{q_1} N_{q_2} \cdots N_{q_{r(m)}} L(E,1) = - \sum_{d=1}^{r(m)-1} \sum_{\substack{n \mid m \\ r(n)=d}} \prod_{q \mid \frac{m}{n}} (1 - q + N_q) \langle n \rangle_{\chi_1} - \langle m \rangle_{\chi_1}.
\end{equation}
By our assumption, and note that $ord_2 (1 - q + N_q) \geq 2$, we then see that
$$
ord_2(\prod_{q \mid \frac{m}{n}} (1 - q + N_q) \langle n \rangle_{\chi_1} / c_f) > ord_2(N_{q_1} N_{q_2} \cdots N_{q_{r(m)}} L(E,1)/ c_f)
$$
holds for all divisors $n>1$ of $m$ with $n \neq m$. Then the assertion for $m=q_1 q_2 \cdots q_{r(m)}$ follows immediately. This completes the proof of the lemma.
\end{proof}

\smallskip

\subsection{Integrality at 2}

For any positive odd square-free integer $m$ with $(m,C)=1$, we write $m=m^{+}m^{-}$, where all the prime factors of $m^{+}$ are congruent to $1$ modulo $4$, and all the prime factors of $m^{-}$ are congruent to $3$ modulo $4$. We have the following result of integrality at $2$, which plays an important role to carry out our induction arguments in the proof of our main theorems.

\begin{prop}\label{Integrality}
Let $E$ be an optimal elliptic curve over $\BQ$ with conductor $C$. Let $m$ be any integer of the form $m=m^{+}m^{-}=q_1 q_2 \cdots q_{r(m)}$ with $(m,C)=1$, and $q_1, \ldots, q_{r(m)}$ arbitrary distinct odd primes. Then we have 
$$ 
\sum_{\substack{d \mid m \\ 2 \mid r(d^{-})}} \langle m \rangle_{\chi_d}/{c_f} = 2^{r(m)-1} \Psi_m; \ \ \ \ \sum_{\substack{d \mid m \\ 2 \nmid r(d^{-})}} \langle m \rangle_{\chi_d}/{i{c^-_f}} = 2^{r(m)-1} \Psi'_m,
$$
where $d=d^{+}d^{-}$, $r(d^{-})$ is the number of prime factors of $d^{-}$, and $\Psi_m, \Psi'_m$ are integers. Moreover, $\Psi_m$ and $\Psi'_m$ are integers of the same parity when $\Delta_E < 0$.
\end{prop}
\begin{proof}
By \eqref{T'_{d,m}}, \eqref{S'm}, \eqref{Tm} and \eqref{Tm-2}, we have
\begin{equation}\label{T'_{d,m}even}
\sum_{\substack{d \mid m \\ 2 \mid r(d^{-})}} \langle m \rangle_{\chi_d} = \langle m \rangle_{\chi_1} + \sum_{\substack{d \mid m, 2 \mid r(d^{-}) \\ 1< d < m}} \ \prod_{q \mid \frac{m}{d}}(a_q - 2 \chi_d(q)) \cdot \langle d \rangle_{\chi_d} + \frac{1+(-1)^{r(m^{-})}}{2} \, \langle m \rangle_{\chi_m} = s \cdot c_f; 
\end{equation}
\begin{equation}\label{T'_{d,m}odd}
\sum_{\substack{d \mid m \\ 2 \nmid r(d^{-})}} \langle m \rangle_{\chi_d} = \sum_{\substack{d \mid m, 2 \nmid r(d^{-}) \\ 1< d < m}} \ \prod_{q \mid \frac{m}{d}}(a_q - 2 \chi_d(q)) \cdot \langle d \rangle_{\chi_d} + \frac{1-(-1)^{r(m^{-})}}{2} \, \langle m \rangle_{\chi_m} = t \cdot i{c^-_f},
\end{equation}
where $s,t$ are integers.

On the other hand, by the definition of $\langle m \rangle_{\chi_d}$ and in view of  \eqref{k/mf}, we have 
\begin{equation}\label{T'_{d,m}_2}
\sum_{d \mid m} \langle m \rangle_{\chi_d} = 2^{r(m)} \mathop{{\sum}^*}_{k \in (\BZ/m\BZ)^{\times}} \langle \{0,\frac{k}{m}\}, f \rangle = 2^{r(m)-1}\mathop{{\sum}^*}_{k \in (\BZ/m\BZ)^{\times}} (s_{k,m} c_f + i t_{k,m} {c^-_f}),
\end{equation}
where $\mathop{{\sum}^*}$ means that $k$ runs over all the elements of $(\BZ/m\BZ)^{\times}$ such that $\chi_{q_i}(k)=1$ for all $1 \leq i \leq r(m)$. Then combining \eqref{T'_{d,m}_2} with \eqref{T'_{d,m}even} and \eqref{T'_{d,m}odd}, it follows that 
$$
s = 2^{r(m)-1} \Psi_m := 2^{r(m)-1} \mathop{{\sum}^*}_{k \in (\BZ/m\BZ)^{\times}} s_{k,m};
$$
$$
t = 2^{r(m)-1} \Psi'_m := 2^{r(m)-1} \mathop{{\sum}^*}_{k \in (\BZ/m\BZ)^{\times}} t_{k,m}.
$$
The last assertion of the proposition holds since $s_{k,m}, t_{k,m}$ are integers of the same parity when $\Delta_E < 0$. This completes the proof of the proposition.
\end{proof}

\medskip

\section{Lower Bound}

In order to prove our main theorem, we should understand the behaviour of the Tamagawa factors under twisting, for example, one can see \cite[\S7]{Coates1}). Recall that $c_q(E^{(M)})$ denotes the Tamagawa factor at the prime $q \mid M$. Since $(q,2C)=1$, reduction modulo $q$ on $E$ gives an isomorphism $E(\BQ_q)[2] \cong E(\BF_q)[2]$. Hence, we have 
$$
ord_2(c_q(E^{(M)})) = ord_2(\# E^{(M)}(\BQ_q)[2]) = ord_2(\# E(\BQ_q)[2]) = ord_2(\# E(\BF_q)[2])
$$
by \cite[Lemma 37]{Coates1}. Recall that $ord_2(|E(\BF_q)|)=ord_2(N_q)$. It follows that $2 \mid c_q(E^{(M)})$ if and only if $2 \mid N_q$. Hence, we have   
$$
t(m):=t_E(M)=\sum_{\substack{q \mid m \\ 2 \mid c_q(E^{(M)})}} 1=\sum_{\substack{q \mid m \\ 2 \mid N_q}} 1,
$$
where $m=|M|$. It is plain that 
$$
t(m)=t(m_1)+t(m_2)
$$ 
for $m=m_1 m_2$, and that $t(m)=r(m)$ if $|E(\BQ)_{\tor}|$ is even.

\begin{lem}\label{ord_2 S'_m'1}
Let $E$ be the optimal elliptic curve over $\BQ$ attached to $f$. Let $m=q_1 q_2 \cdots q_{r(m)}$ be a product of $r(m)$ distinct odd primes with $(m,C)=1$, $r(m) \geq 1$. We have
$$
ord_2 (\langle m \rangle_{\chi_1} / c_f) \geq 
\left\{
\begin{array}{ll}
r(m)-1             & \hbox{if $|E(\BQ)_{\tor}|$ is even;} \\
t(m)                & \hbox{if $|E(\BQ)_{\tor}|$ is odd.}
\end{array}
\right.
$$
\end{lem}
\begin{proof}
When $r(m)=1$, in view of \eqref{N_qL(E,1)}, the lemma holds obviously if $|E(\BQ)_{\tor}|$ is even. If $|E(\BQ)_{\tor}|$ is odd, there exists a prime $p$ coprime to $2C$ such that $N_p$ is odd, so $ord_2(L(E,1) / {c_f}) \geq 0$. It follows that 
$$
ord_2 (\langle q \rangle_{\chi_1} / c_f) \geq ord_2(N_q) \geq t(q).
$$
%When $r(m)=2$, say $m=q_1 q_2$, by \eqref{NqL(E,1)}, we have that
%$$
%\langle q_1 q_2 \rangle_{\chi_0} = N_{q_1} N_{q_2} L(E,1) - (1 - q_2) \langle q_1 \rangle_{\chi_0} - (1 - q_1) \langle q_2 \rangle_{\chi_0}. 
%$$
%If $|E(\BQ)_{\tor}|$ is even, the lemma holds since $ord_2(N_q) \geq 1$ for any $q \mid m$. If $|E(\BQ)_{\tor}|$ is odd, we have $0 \leq t(q) \leq 1$ for any $q \mid m$. Moreover, it is easy to see that $ord_2((1 - q_2) \langle q_1 \rangle_{\chi_0}/c_f)\geq t(q_2)+t(q_1)$, $ord_2((1 - q_1) \langle q_2 \rangle_{\chi_0}/c_f)\geq t(q_1)+t(q_2)$, and $ord_2(N_{q_1} N_{q_2}) \geq t(q_1 q_2)$. It follows that 
%$$
%ord_2 (\langle q_1 q_2 \rangle_{\chi_0} / c_f) \geq t(q_1 q_2).
%$$
We now prove the lemma by induction on $r(m)$. Assume $r(m) > 1$, and that the lemma is true for all divisors $n>1$ of $m$ with $n \neq m$. For $m=q_1 q_2 \cdots q_{r(m)}$, in view of \eqref{NqL(E,1)}, we have that
$$
\langle m \rangle_{\chi_1} = (-1)^{r(m)} N_{q_1} N_{q_2} \cdots N_{q_{r(m)}} L(E,1) - \sum_{j=1}^{r(m)-1} \sum_{\substack{n \mid m \\ r(n)=j}} \prod_{q \mid \frac{m}{n}} (1 - q) \langle n \rangle_{\chi_1}.
$$
By our assumption, we see that
$$
ord_2(\prod_{q \mid \frac{m}{n}} (1 - q) \langle n \rangle_{\chi_1} / c_f) \geq 
\left\{
\begin{array}{ll}
r(\frac{m}{n})+r(n)-1=r(m)-1             & \hbox{if $|E(\BQ)_{\tor}|$ is even;} \\
t(\frac{m}{n})+t(n)=t(m)                   & \hbox{if $|E(\BQ)_{\tor}|$ is odd.}
\end{array}
\right.
$$
Same result holds for $ord_2(N_{q_1} N_{q_2} \cdots N_{q_{r(m)}} L(E,1) / c_f)$. Hence, the assertion for $m=q_1 q_2 \cdots q_{r(m)}$ follows. This completes the proof of the lemma.
\end{proof}

\begin{thm}[Theorem \ref{ThmLowerBound}]\label{LowerBound-1}
Let $E$ be an optimal elliptic curve over $\BQ$ with conductor $C$. Let $M=\epsilon q_1 q_2 \cdots q_r$ be a product of $r$ odd distinct primes with $(M,C)=1$, where the sign $\epsilon = \pm 1$ is chosen so that $M \equiv 1 \mod 4$. We then have
$$
ord_2(L(E^{(M)},1)/c_\infty(E^{(M)})) \geq t_E(M)-1-ord_2(\nu_E). 
$$
\end{thm}
\begin{proof}
We denote $m=m^{+}m^{-}=q_1 q_2 \cdots q_{r(m)}$, and 
$$
c^{\pm}_\infty(E)=
\left\{
\begin{array}{ll}
c_\infty(E)             & \hbox{if $m \equiv 1 \mod 4$;} \\
ic^-_\infty(E)         & \hbox{if $m \equiv 3 \mod 4$,}
\end{array}
\right.
\text{ and \ \ }
c^{\pm}_f=
\left\{
\begin{array}{ll}
c_f             & \hbox{if $m \equiv 1 \mod 4$;} \\
ic^-_f         & \hbox{if $m \equiv 3 \mod 4$.}
\end{array}
\right.
$$
Since 
$$
ord_2(L(E^{(M)},1)/c_\infty(E^{(M)})) = ord_2(\sqrt{M} L(E^{(M)},1) / c^{\pm}_\infty(E)) = ord_2 (\langle m \rangle_{\chi_m}/c^{\pm}_f) - ord_2(\nu_E),
$$
we only need to work on $\langle m \rangle_{\chi_m}/c^{\pm}_f$. We shall prove this lemma by induction on $r(m)$. 

When $r(m)=1$, say $m=q$, Proposition \ref{Integrality} shows that 
$$ 
(\langle q \rangle_{\chi_1}+\langle q \rangle_{\chi_q})/{c_f} = \Psi_q \ (\text{if } q \equiv 1 \mod 4);
$$
$$
\langle q \rangle_{\chi_q}/i{c^-_f} = \Psi'_q \ (\text{if } q \equiv 3 \mod 4), 
$$
then the lemma follows immediately by Lemma \ref{ord_2 S'_m'1}. 
%When $r(m)=2$, say $m=q_1 q_2$, Proposition \ref{Integrality} shows that
%$$
%\langle q_1 q_2 \rangle_{\chi_0} + \sum_{qq'=q_1 q_2} \frac{1+(-1)^{\frac{q-1}{2}}}{2} \, (a_{q'} - 2 \chi_{q}(q')) \cdot T_{q} + \frac{1+(-1)^{\frac{q_1 q_2-1}{2}}}{2} \, \langle q_1 q_2 \rangle_{\chi_{q_1 q_2}} = 2\Psi_{q_1 q_2} \cdot c_f; 
%$$
%$$
%\sum_{qq'=q_1 q_2} \frac{1-(-1)^{\frac{q-1}{2}}}{2} \, (a_{q'} - 2 \chi_{q}(q')) \cdot T_{q} + \frac{1-(-1)^{\frac{q_1 q_2-1}{2}}}{2} \, \langle q_1 q_2 \rangle_{\chi_{q_1 q_2}} = 2\Psi'_{q_1 q_2} \cdot i{c^-_f}.
%$$
%Note that $\ord_2 ((a_{q'} - 2 \chi_{q}(q')) \cdot T_{q}/c^{\pm}_f) \geq t(q')+t(q)-1=t(q_1 q_2)-1$, and in view of the above two equations, it follows that $ord_2(\langle q_1 q_2 \rangle_{\chi_{q_1 q_2}}/c^{\pm}_f) \geq t(q_1 q_2)-1$ again by Lemma \ref{ord_2 S'_m'1}. 

Now assume that $r(M) > 1$, and that the theorem is true for all divisors $n \equiv 1 \mod 4$ of $M$ with $n \neq M$, so we have that 
$$
ord_2(\langle d \rangle_{\chi_d}/c^{\pm}_f) \geq t(d) -1,
$$ 
where $d \mid M$ and $1 < d < q_1 q_2 \cdots q_r:=m$. Again by Proposition \ref{Integrality}, we have 
$$
\langle m \rangle_{\chi_1} + \sum_{\substack{d \mid m, 2 \mid r(d^{-}) \\ 1< d < m}} \ \prod_{q \mid \frac{m}{d}}(a_q - 2 \chi_d(q)) \cdot \langle d \rangle_{\chi_d} + \langle m \rangle_{\chi_m} = 2^{r-1} \Psi_m \cdot c_f \ \ (\text{if } 2 \mid r(m^-));
$$
$$
\sum_{\substack{d \mid m, 2 \nmid r(d^{-}) \\ 1< d < m}} \ \prod_{q \mid \frac{m}{d}}(a_q - 2 \chi_d(q)) \cdot \langle d \rangle_{\chi_d} + \langle m \rangle_{\chi_m} = 2^{r-1} \Psi'_m \cdot i{c^-_f} \ \ (\text{if } 2 \nmid r(m^-)).
$$
Note that 
$$
\ord_2 (\prod_{q \mid \frac{m}{d}}(a_q - 2 \chi_d(q)) \cdot \langle d \rangle_{\chi_d}/c^{\pm}_f) \geq t(\frac{m}{d})+t(d)-1=t(m)-1,
$$
and in view of the above two equations, it follows that 
$$
ord_2(\langle m \rangle_{\chi_m}/c^{\pm}_f) \geq t(m)-1
$$ 
again by Lemma \ref{ord_2 S'_m'1}. This completes the proof the theorem.
\end{proof}

\medskip

\section{Non-vanishing result}

In this section, we shall verify Conjecture \ref{Conj} for elliptic curves with only one non-trivial rational $2$-torsion point. In particular, we will show that in some sub-families of quadratic twists of certain elliptic curves, there exists an explicit infinite family of quadratic twists with analytic rank $0$. Of course, we can use the precise $L$-values to verify the $2$-part of the Birch and Swinnerton-Dyer conjecture, combining with results given by the classical $2$-descents, for example, in \cite[Chapter \uppercase\expandafter{\romannumeral10}]{Silverman}. We first prove the following lemma, which admits stronger lower bound than the conclusion in Lemma \ref{ord_2 S'_m'1}.

\begin{lem}\label{ord_2 S'_m'}
Let $E$ be the optimal elliptic curve over $\BQ$ with analytic rank zero attached to $f$. Let $m$ be any integer of the form $m=m^{+}m^{-}=q_1 q_2 \cdots q_{r(m)}$, with $(m,C)=1$, $r(m^{+}) \geq 1$, $r(m^{-}) \geq 1$, and $q_1, \ldots, q_{r(m)}$ arbitrary distinct odd primes. If $ord_2 (N_{q_i}) = 1$ for any $1 \leq i \leq r(m)$, then we have
$$
ord_2 (\langle m \rangle_{\chi_1} / c_f) > r(m)-1.
$$
\end{lem}
\begin{proof}
We first prove the case when $r(m)=2$, say $m=m^{+}m^{-}=q_1 q_2$ with $q_1 \equiv 1 \mod 4$, $q_2 \equiv 3 \mod 4$. By \eqref{NqL(E,1)}, we have that
$$
N_{q_1} N_{q_2} L(E,1) =  (1 - q_2) \langle q_1 \rangle_{\chi_1} + (1 - q_1) \langle q_2 \rangle_{\chi_1} + \langle q_1 q_2 \rangle_{\chi_1}. 
$$
In view of \eqref{ms1}, we have 
$$
ord_2 (S_{q_i}/c_f)=ord_2(N_{q_i} L(E,1)/c_f) \ \ (i=1,2),
$$
it follows that 
$$
ord_2 (N_{q_1} N_{q_2} L(E,1)/c_f) = ord_2 ((1 - q_2) \langle q_1 \rangle_{\chi_1}/c_f) < ord_2 ((1 - q_1) \langle q_2 \rangle_{\chi_1}/c_f).
$$
Thus, we must have 
$$
ord_2 (\langle q_1 q_2 \rangle_{\chi_1}/c_f) > ord_2 (N_{q_1} N_{q_2} L(E,1)/c_f).
$$

We now prove this lemma by induction on $r(m)$. Assuming that $r(m) > 2$, and that the lemma is true for all divisors $n=n^{+}n^{-}>1$ of $m$ with $r(n^{+}) \geq 1$, $r(n^{-}) \geq 1$, and $n \neq m$. Then, for $m=q_1 q_2 \cdots q_{r(m)}$, we have 
\begin{equation}\label{eq211}
\prod_{q \mid m} N_q \cdot L(E,1) = (-1)^{r(m)} (\prod_{q \mid m^{-}} (1 - q) \cdot \langle m^{+} \rangle_{\chi_1} + \sum_{j=1}^{r(m)-1} \sum_{\substack{n \mid m \\ n \neq m^{+} \\ r(n)=j}} \prod_{q \mid \frac{m}{n}} (1 - q) \cdot \langle n \rangle_{\chi_1} + \langle m \rangle_{\chi_1})
\end{equation}
again by \eqref{NqL(E,1)}. Since 
$$
ord_2(\langle m^{+} \rangle_{\chi_1}/c_f) = ord_2 (\prod_{q \mid m^{+}} N_q \cdot L(E,1)/c_f))
$$
by \cite[Lemma 2.2]{Cai} and $ord_2 (1-q) = ord_2 (N_q)$ for $q \mid m^{-}$, it follows that 
\begin{equation}\label{eq212}
ord_2 (\prod_{q \mid m^{-}} (1 - q) \cdot \langle m^{+} \rangle_{\chi_1}/c_f) = ord_2 (\prod_{q \mid m} N_q \cdot L(E,1)/c_f)).
\end{equation}
However, when $n \mid m$ and $n \neq m^{+}$, we claim
\begin{equation}\label{eq213}
ord_2 (\prod_{q \mid \frac{m}{n}} (1 - q) \cdot \langle n \rangle_{\chi_1}/c_f) > ord_2 (\prod_{q \mid m} N_q \cdot L(E,1)/c_f)).
\end{equation}
Indeed, when $m^{+} \mid n \neq m^{+}$, we have 
$$
ord_2 (\langle n \rangle_{\chi_1}/c_f) > ord_2 (\prod_{q \mid n} N_q \cdot L(E,1)/c_f)
$$
by our assumption, the claim then follows immediately since $ord_2 (1-q) = ord_2 (N_q)$ for $q \mid \frac{m}{n}$. When $m^{+} \nmid n$, we have 
$$
ord_2 (\langle n \rangle_{\chi_1}/c_f) \geq ord_2 (\prod_{q \mid n} N_q \cdot L(E,1)/c_f)
$$
by our assumption, and by \cite[Lemma 2.2]{Cai} and Corollary \ref{N_qL(E,1)m}, where the equality holds when $n \mid m^{+}$ or $n \mid m^{-}$. Since $m^{+} \nmid n$, there must exist at least one prime factor of $\frac{m}{n}$ which is congruent to $1$ modulo $4$, it follows that 
$$
ord_2 (\prod_{q \mid \frac{m}{n}} (1-q)) > ord_2 (\prod_{q \mid \frac{m}{n}} N_q).
$$ 
Thus, the claim is true for both cases. In view of \eqref{eq211}, and combining with \eqref{eq212} and \eqref{eq213}, it easily follows that 
$$
ord_2 (\langle m \rangle_{\chi_1} / c_f) > ord_2 (\prod_{q \mid m} N_q \cdot L(E,1) / c_f) \geq r(m)-1.
$$
This completes the proof of the lemma. 
\end{proof}

When $E$ has only one rational $2$-torsion point and no rational cyclic $4$-isogeny, recall that, if $\Delta_E<0$, $\mathcal{S}$ is an infinite set of odd primes
$$
\mathcal{S}=\{q \ | \ ord_2(|E(\BF_q)|)=ord_2(|E(\BQ)[2]|)\},
$$
whence we have $ord_2 (|E(\BF_q)|) = ord_2 (N_q) = 1$ for any $q \in \mathcal{S}$.
\begin{thm}\label{MainThm-1}
Let $E$ be an optimal elliptic curve over $\BQ$ with conductor $C$. Assume that 
\begin{enumerate}
  \item $\Delta_E < 0$;
  \item $E(\BQ)[2] \cong \BZ/2\BZ$;
  \item $E$ has odd Manin constant;
  \item $ord_2 (L(E,1)/c_\infty(E))=-1$.
\end{enumerate}
Let $M=\epsilon q_1 q_2 \cdots q_r$ be a product of $r$ distinct primes in $\mathcal{S}$, where the sign $\epsilon = \pm 1$ is chosen so that $M \equiv 1 \mod 4$. Then $L(E^{(M)},1) \neq 0$, and we have
$$
ord_2(L(E^{(M)},1)/c_\infty(E^{(M)}))=r-1.
$$
In particular, $E^{(M)}(\BQ)$ and $\Sha(E^{(M)})$ are both finite.
\end{thm}
\begin{proof}
As usual, we write $M=\epsilon M^{+}M^{-}$, where all the prime factors of $M^{+}$ are congruent to $1$ modulo $4$, and all the prime factors of $M^{-}$ are congruent to $3$ modulo $4$. The theorem has been proved when $M=M^{+}$ (see \cite[Theorem 1.1]{Cai}), but for the remaining cases, the proof is very different, since the imaginary period gets involved. We now need to prove the case when $r(M^{+}) \geq 0$ and $r(M^{-}) \geq 1$. 

We first look at the case $r(M^{+})=0$ and $r(M^{-})=1$, i.e., $M=-q\equiv 1 \mod 4$. In view of Proposition \ref{m_chi} and note the $\Delta<0$, we have
$$
L(E^{(-q)},1)/c_\infty(E^{(-q)}) \equiv \langle q \rangle_{\chi_q}/i{c^-_f} \equiv \langle q \rangle_{\chi_1}/{c_f} \equiv N_q L(E,1)/c_f \equiv 1 \mod 2.
$$
Thus the theorem holds for this initial case. When $r(M)=2$, there are two cases left to be done, i.e., $r(M^{+})+2=r(M^{-})=2$ and $r(M^{+})=r(M^{-})=1$. 

We now consider the case $r(M^{+})=0$ and $r(M^{-})=2$, say $M=q_1 q_2$ with $q_1 \equiv q_2 \equiv 3 \mod 4$. By Proposition \ref{Integrality} we have that  
$$
(\langle q_1 q_2 \rangle_{\chi_1} + \langle q_1 q_2 \rangle_{\chi_{q_1 q_2}})/c_f = 2 \Psi_{q_1 q_2}; 
$$
$$
((a_{q_1} - 2 \chi_{q_2}({q_1})) \langle q_2 \rangle_{\chi_{q_2}} + (a_{q_2} - 2 \chi_{q_1}({q_2}))\langle q_1 \rangle_{\chi_{q_1}})/i{c^-_f} = 2\Psi'_{q_1 q_2}.
$$
Since 
$$a_{q_1} - 2 \chi_{q_2}({q_1}) \equiv a_{q_2} - 2 \chi_{q_1}({q_2}) \equiv 0 \mod 4,
$$ 
it follows that  $\Psi'_{q_1 q_2}$ must be even, and so is $\Psi_{q_1 q_2}$ by the last assertion of Lemma \ref{Integrality}. Not that 
$$
ord_2(\langle q_1 q_2 \rangle_{\chi_1}/c_f)=1
$$ 
by Corollary \ref{N_qL(E,1)m}, we must have
$$
ord_2(\langle q_1 q_2 \rangle_{\chi_{q_1 q_2}}/c_f)=1.
$$ 
This proves the case $r(M^{+})=0$ and $r(M^{-})=2$. 

We next consider the case $r(M^{+})=r(M^{-})=1$, say $M=-q_1 q_2$ with $q_1 \equiv 1 \mod 4$ and $q_2 \equiv 3 \mod 4$. By Proposition \ref{Integrality} we have that 
$$
(\langle q_1 q_2 \rangle_{\chi_1} + (a_{q_2} - 2 \chi_{q_1}({q_2}))\langle q_1 \rangle_{\chi_{q_1}})/c_f = 2 \Psi_{q_1 q_2}; 
$$
$$
 ((a_{q_1} - 2 \chi_{q_2}({q_1})) \langle q_2 \rangle_{\chi_{q_2}} + \langle q_1 q_2 \rangle_{\chi_{q_1 q_2}})/i{c^-_f} = 2\Psi'_{q_1 q_2}.
$$
Note that we have proved 
$$
ord_2(\langle q_1 \rangle_{\chi_{q_1}}/c_f )=ord_2(\langle q_2 \rangle_{\chi_{q_2}}/i{c^-_f})=0,
$$ 
then 
$$
ord_2((a_{q_2} - 2 \chi_{q_1}({q_2}))\langle q_1 \rangle_{\chi_{q_1}}/c_f) > 1,$$ 
and 
$$
ord_2(\langle q_1 q_2 \rangle_{\chi_1}/c_f) > 1
$$ 
by Lemma \ref{ord_2 S'_m'}, it follows that $\Psi_{q_1 q_2}$ must be even, and so is $\Psi'_{q_1 q_2}$. But\ $$
ord_2((a_{q_1} - 2 \chi_{q_2}({q_1})) \langle q_2 \rangle_{\chi_{q_2}}/i{c^-_f})=1,
$$ 
hence 
$$ord_2(\langle q_1 q_2 \rangle_{\chi_{q_1 q_2}}/i{c^-_f})=1.$$ 
This proves the case $r(M^{+})=r(M^{-})=1$.

We are now ready to prove the theorem by induction on $r(M)$. Now assume that $r(M) > 2$, and that the theorem is true for $r(M)<r$. We shall prove the theorem in two cases. 

When $r(M^{-})$ is even, by Proposition \ref{Integrality}, we have 
\begin{equation}\label{eq311}
\langle M^{+}M^{-} \rangle_{\chi_1} + \sum_{\substack{d \mid {M^{+}M^{-}}, 2 \mid r(d^{-}) \\ 1< d < {M^{+}M^{-}}}} \ \prod_{q \mid \frac{{M^{+}M^{-}}}{d}}(a_q - 2 \chi_d(q)) \cdot \langle d \rangle_{\chi_d} + \langle M^{+}M^{-} \rangle_{\chi_{M^{+}M^{-}}} = 2^{r(M)-1} \Psi_{M^{+}M^{-}} \cdot c_f;
\end{equation}
\begin{equation}\label{eq312}
\sum_{\substack{d \mid {M^{+}M^{-}}, 2 \nmid r(d^{-}) \\ 1< d < {M^{+}M^{-}}}} \ \prod_{q \mid \frac{{M^{+}M^{-}}}{d}}(a_q - 2 \chi_d(q)) \cdot \langle d \rangle_{\chi_d} = 2^{r(M)-1} \Psi'_{M^{+}M^{-}} \cdot i{c^-_f}.
\end{equation}
When $1< d < {M^{+}M^{-}}$, since $4 \mid (a_q - 2 \chi_d(q))$ if $q \mid M^{-}$, we have
\begin{equation}\label{eq315}
ord_2(\prod_{q \mid \frac{{M^{+}M^{-}}}{d}}(a_q - 2 \chi_d(q)) \cdot \langle d \rangle_{\chi_d} / c^{\pm}_f) 
\left\{
\begin{array}{ll}
= r(\frac{{M^{+}M^{-}}}{d})+r(d)-1=r(M)-1        & \hbox{if $M^{-} \mid d$;} \\
> r(\frac{{M^{+}M^{-}}}{d})+r(d)-1=r(M)-1        & \hbox{if $M^{-} \nmid d$,}
\end{array}
\right.
\end{equation}
by our assumption. If $2 \nmid r(d^{-})$, but $r(M^{-})$ is even, we must have $M^{-} \nmid d$. Combining with \eqref{eq312} and \eqref{eq315}, it follows that $\Psi'_{M^{+}M^{-}} $ must be even, and so is $\Psi_{M^{+}M^{-}}$, since $\Delta_E<0$. We now investigate the middle terms of the left-hand side of \eqref{eq311} divided by $c_f$, there are exactly $2^{r(M^{+})}-1$ terms which have $2$-adic valuation $r(M)-1$ and others have $2$-adic valuation greater than $r(M)-1$ in view of \eqref{eq315}. Thus, when $r(M^{+}) = 0$, the summation of all the middle terms divided by $c_f$ has $2$-adic valuation greater than $r(M)-1$. While 
$$
ord_2(\langle M^{-} \rangle_{\chi_1}/c_f)=r(M^{-})-1
$$ 
by Corollary \ref{N_qL(E,1)m}, we must have 
$$ord_2(\langle M^{-} \rangle_{\chi_{M^{-}}}/c_f)=r(M^{-})-1.$$
When $r(M^{+}) \neq 0$, the summation of all the middle terms divided by $c_f$ has $2$-adic valuation exactly $r(M)-1$. But $\Psi_{M^{+}M^{-}}$ is even, and 
$$
ord_2(\langle M^{+}M^{-} \rangle_{\chi_1}/c_f)>r(M)-1
$$ 
by Lemma \ref{ord_2 S'_m'}, it follows that we must have 
$$
ord_2(\langle M^{+}M^{-} \rangle_{\chi_{M^{+}M^{-}}}/c_f)=r(M)-1.
$$ 

When $r(M^{-})$ is odd, by Proposition \ref{Integrality}, we have 
\begin{equation}\label{eq313}
\langle M^{+}M^{-} \rangle_{\chi_1} + \sum_{\substack{d \mid {M^{+}M^{-}}, 2 \mid r(d^{-}) \\ 1< d < {M^{+}M^{-}}}} \ \prod_{q \mid \frac{{M^{+}M^{-}}}{d}}(a_q - 2 \chi_d(q)) \cdot \langle d \rangle_{\chi_d} = 2^{r(M)-1} \Psi_{M^{+}M^{-}} \cdot c_f;
\end{equation}
\begin{equation}\label{eq314}
\sum_{\substack{d \mid {M^{+}M^{-}}, 2 \nmid r(d^{-}) \\ 1< d < {M^{+}M^{-}}}} \ \prod_{q \mid \frac{{M^{+}M^{-}}}{d}}(a_q - 2 \chi_d(q)) \cdot \langle d \rangle_{\chi_d} + \langle M^{+}M^{-} \rangle_{\chi_{M^{+}M^{-}}} = 2^{r(M)-1} \Psi'_{M^{+}M^{-}} \cdot i{c^-_f}.
\end{equation}
If $2 \mid r(d^{-})$, but $r(M^{-})$ is odd, we must have $M^{-} \nmid d$. When $r(M^{+}) = 0$, in view of \eqref{eq315} and \eqref{eq313}, we conclude that $\Psi_{M^{-}}$ must be odd by Corollary \ref{N_qL(E,1)m}, and so is $\Psi'_{M^{-}}$, since $\Delta_E<0$. Similar to the case $2 \mid r(M^{-})$, same argument shows that the first summation in \eqref{eq314} divided by $i{c^-_f}$ has $2$-adic valuation greater than $r(M^{-})-1$. Thus, we must have 
$$
ord_2(\langle M^{-} \rangle_{\chi_{M^{-}}}/i{c^-_f})=r(M^{-})-1.
$$
When $r(M^{+}) \neq 0$, in view of \eqref{eq315} and \eqref{eq313}, we conclude that $\Psi_{M^{+}M^{-}}$ must be even by Lemma \ref{ord_2 S'_m'}, and so is $\Psi'_{M^{+}M^{-}}$. Still as before, the first summation in \eqref{eq314} divided by $i{c^-_f}$ has $2$-adic valuation exactly $r(M)-1$. Therefore, we have 
$$
ord_2(\langle M^{+}M^{-} \rangle_{\chi_{M^{+}M^{-}}}/i{c^-_f})=r(M)-1.
$$ 

This completes the proof for both cases in the induction argument. Since $\nu_E$ is odd, it follows that
$$
ord_2(L(E^{(M)},1)/c_\infty(E^{(M)})) = ord_2 (\langle M^{+}M^{-} \rangle_{\chi_{M^{+}M^{-}}}/c^{\pm}_f) = r-1.
$$
This completes the proof of the theorem combining the celebrated theorems of Gross--Zagier \cite{Gross} and Kolyvagin \cite{Kolyvagin}.
\end{proof}

\bigskip

\noindent Shuai Zhai, {\it Research Center for Mathematics and Interdisciplinary Sciences, Shandong University, Qingdao, Shandong, China.} 

\medskip

\noindent {\it E-mail:} zhai@sdu.edu.cn

\end{document}